\documentclass[10pt, a4paper]{amsart}
 
\usepackage[pdftex]{graphicx}
\usepackage{amsfonts}
\usepackage{amssymb}
\usepackage{algorithmic}
\usepackage{array}
\usepackage{url}
\usepackage{amsmath}
\usepackage[english]{babel}
\usepackage{mathtools}
\usepackage{enumerate}
\usepackage[numbers]{natbib}
\usepackage[hidelinks]{hyperref}
\usepackage{tikz}
\usepackage{tkz-graph}
\usepackage{xcolor}

\newcommand{\set}[1]{\left\{ #1 \right\}}

\newcommand{\isom}{\xrightarrow{\sim}}
\newcommand{\Id}{\mathrm{Id}}
\newcommand{\N}{\mathbb{N}}
\newcommand{\Z}{\mathbb{Z}}

\newcommand{\F}{\mathbb{F}}
\newcommand{\A}{\mathbb{A}}
\newcommand{\PP}{\mathbb{P}}
\newcommand{\GL}{\mathrm{GL}}

\newcommand{\PGL}{\mathrm{PGL}}

\newtheorem{theorem}{Theorem}[section]
\newtheorem{proposition}[theorem]{Proposition}

\newtheorem{lemma}[theorem]{Lemma}

\theoremstyle{definition}

\newtheorem{example}{Example}[section]

\theoremstyle{remark}
\newtheorem{remark}{Remark}
\newtheorem{question}{Question}

\numberwithin{equation}{section}

\begin{document}

\title{Fractional jumps: complete characterisation and an explicit infinite family}

%%%%%%%%%%%%%%%%%%%%%%%%%%%%%%%%%%%%%%%%%%%%%%%%%%%
%AUTHORS
%%%%%%%%%%%%%%%%%%%%%%%%%%%%%%%%%%%%%%%%%%%%%%%%%%%

\author{Federico Amadio Guidi}
\address{Mathematical Institute, University of Oxford, Oxford, UK}
\email{federico.amadio@maths.ox.ac.uk}
\author{Giacomo Micheli}
\address{Mathematical Institute, University of Oxford, Oxford, UK}
\email{giacomo.micheli@maths.ox.ac.uk}

%%%%%%%%%%%%%%%%%%%%%%%%%%%%%%%%%%%%%%%%%%%%%%%%%%%
%%%%%%%%%%%%%%%%%%%%%%%%%%%%%%%%%%%%%%%%%%%%%%%%%%%

\maketitle

\begin{abstract}
%In this paper we complete the theory of fractional jumps by showing the uniqueness of the construction, i.e. transitive fractional jumps can only arise from transitive projective automorphisms. In addition, we prove that such construction is always feasible by providing an explicit infinite class of projectively primitive polynomials whose companion matrix can be used to define a full orbit sequence over an affine space. %computational advantages

In this paper we provide a complete characterisation of transitive fractional jumps by showing that they can only arise from transitive projective automorphisms. Furthermore, we prove that such construction is feasible for arbitrarily large dimension by exhibiting an infinite class of projectively primitive polynomials whose companion matrix can be used to define a full orbit sequence over an affine space.
\end{abstract}
%{\footnotesize\noindent\textbf{Keywords}: full orbit sequences, inversive congruential generators, projective space, affine space.}
%{\footnotesize\noindent\textbf{MSC2010 subject classification}: 11T06, 11B37, 15B33, 11B50.}

\section{Introduction}

The study of dynamical systems over finite fields have a long history (see for example \cite{bib:BW05, bib:chou95, bib:eich91, bib:EMG09, bib:GSW03, bib:NS02, bib:NS03, bib:TW06}) and is an interesting and still hot topic (see for example \cite{bib:FMS16, bib:GPOS14, bib:HBM17, bib:ost10, bib:OPS10, bib:OS10degree, bib:OS10length, bib:winterhof10}), both for its number theoretical impact in finite fields theory, and for its practical applications, in particular for random number generation. %. Therefore, it has always been a subject of great interest in applied algebra, in particular for cryptographers and finite field theorists.

Let $q$ be a prime power, let $\F_q$ denote the finite field with $q$ elements, and let $m$ be a positive integer. One of the most interesting questions for applications consists of constructing sequences over the $m$-dimensional affine space over $\F_q$ defined by iterations of rational maps $f:\F_q^m \rightarrow \F_q^m$ satisfying the following conditions:
\begin{enumerate}
\item The period of the recursive sequence $\{f^k(0)\}_{k\in \N}$ they define is ``long''.
\item Their iterations as rational maps have ``low degree growth''.
\end{enumerate}
The motivation for (1) is rather clear: since we generally want to use these sequences for pseudorandom number generation, we do not want to revisit an element twice too soon, or otherwise the entire sequence will repeat. The motivation for (2) is a little more subtle and comes from the uniformity conditions we want the sequence to satisfy (for additional information on this see \cite{bib:OS10degree}).

%fractional jumps
In \cite{bib:AGM17} we introduced the theory of fractional jumps to address this problem by showing a natural way to build full orbit sequences from projective automorphisms, recovering as a particular case the construction of the Inversive Congruential Generator.

In  this paper we complete the theory of fractional jumps by both proving the uniqueness of the construction, i.e. transitive fractional jumps can only arise from transitive projective automorphisms (except from a couple of degenerate cases which we entirely classify), and by providing an explicit infinite class of projectively primitive polynomials, see definition \cite[Definition 3.1]{bib:AGM17}, whose companion matrix can be used to define a full orbit sequence over $\mathbb{F}_p^{p-1}$, for $p$ a prime.
For this family of fractional jumps, which we call \emph{Artin-Schreier fractional jumps}, we show that the computation of the $(k+1)$-th affine point of the full orbit sequence they define, given the $k$-th one, is as expensive as reading out a look-up table once for each entry.

This latter construction entirely addresses points (1) and (2) above, since the corresponding sequences have full orbit (they cover the entire affine space) and they have zero degree growth. The main technique we use is the fractional jump construction provided in \cite{bib:AGM17}.

%%%%%%%%%%%%%%%%%%%%%%%%%%%%%%%%%%%%%%%%%%%%%%%%%%%
%%%%%%%%%%%%%%%%%%%%%%%%%%%%%%%%%%%%%%%%%%%%%%%%%%%

\subsection{Notation}

We denote by $\N$ the set of natural numbers, and by $\Z$ the set of integers. Given $a \in \Z$, we let $\Z_{\geq a}$ denote the set of integers $k \in \Z$ such that $k \geq a$.

Given a commutative ring with unity $R$, we let $R^*$ be the (multiplicative) group of invertible elements in $R$.

For a prime power $q$, we denote by $\F_q$ the finite field of cardinality $q$. For $m \in \N$, we denote the $m$-dimensional affine space $\F_q^m$ by $\A^m$, and the $m$-dimensional projective space over $\F_q$ by $\PP^m$. More generally, for any vector space $V$ over $\F_q$ we denote by $\PP V$ the projectivisation of $V$. Also, we denote by $\F_q [x_1, \ldots, x_m]$ the ring of polynomials in $m$ variables with coefficients in $\F_q$. %For $v \in \F_q^{m}$, and $i \in \set{1, \ldots, m}$, we let $(v)_i$ denote the $i$-th component of $v$.

For $m \in \N$, let us denote by $\GL_{m}(\F_q)$ the general linear group over $\F_q$, that is the group of $m \times m$ invertible matrices with entries in $\F_q$. Also, we denote by $\PGL_{m} (\F_q)$ the group of automorphisms of $\PP^{m-1}$. Recall that $\PGL_{m} (\F_q)$ can be identified with the quotient group $\GL_{m} (\F_q) / \F_q^* \mathrm{Id}_m$, where $\F_q^* \mathrm{Id}$ is the subgroup of $\F_q^*$-multiples of the identity matrix $\mathrm{Id}_m$. For $M \in \GL_{m} (\F_q)$, we denote by $[M]$ its class in $\PGL_{m} (\F_q)$.

We say that a polynomial $\chi (T) \in \F_q [T]$ of degree $\deg \chi (T) = d$ is \emph{projectively primitive} if it is irreducible and if given any root $\alpha$ in $\F_{q^d} \cong \F_q[T] / (\chi (T))$ the class $\overline{\alpha}$ of $\alpha$ in the quotient group $G = \F_{q^d}^* / \F_q^*$ generates $G$.

Let $X$ be a set, and let $G$ be a group acting on it. For any $x \in X$ we denote by $\mathcal{O}_G (x)$ the orbit of $x$ with respect to the action of $G$ on $X$. Given a bijective map $f : X \rightarrow X$, for any $x \in X$ we set $\mathcal{O}_f (x) = \mathcal{O}_{\langle f \rangle} (x)$, where $\langle f \rangle$ denotes the cyclic subgroup of the group of maps from $X$ to itself generated by $f$, and we define $o_f (x) = |\mathcal{O}_f (x)|$. We say that a bijective map $f : X \rightarrow X$ \emph{acts transitively} on $X$, or simply that it is \emph{transitive}, if for any $x, y \in X$ there exists $k \in \Z$ such that $y = f^k (x)$. Equivalently, $f$ acts transitively on $X$ if and only if for any $x_0 \in X$, the $f$-orbit of $x_0$ has size $o_f (x_0) = |X|$. Finally, we say that a sequence $\set{x_k}_{k \in \N}$ in $X$ has \emph{full orbit} if $\set{x_k \, : \, k \in \N} = X$.

%\[ o_f (x) = \min \set{k \geq 1 \, : \, f^k (x) = x }. \]
 
%%%%%%%%%%%%%%%%%%%%%%%%%%%%%%%%%%%%%%%%%%%%%%%%%%%
%%%%%%%%%%%%%%%%%%%%%%%%%%%%%%%%%%%%%%%%%%%%%%%%%%%

\section{Transitive fractional jumps}

For the sake of completeness, we recall the definition of fractional jump of a projective automorphism, as introduced in \cite{bib:AGM17}.

Fix the standard projective coordinates $X_0, \ldots, X_n$ on $\PP^n$, and fix the canonical decomposition 
\begin{equation*} %\label{decomposition}
\PP^n = U \cup H,
\end{equation*}
where
\begin{align*}
U &= \set{[X_0: \ldots: X_n] \in \PP^n \, : \, X_n \neq 0}, \\
H &= \set{[X_0: \ldots: X_n] \in \PP^n \, : \, X_n = 0}.
\end{align*}

Fix also the isomorphism
\begin{equation*}
\pi : \A^n \isom U, \quad (x_1, \ldots, x_n) \mapsto [x_1: \ldots: x_n: 1].
\end{equation*}

Let now $\Psi$ be an automorphism of $\PP^n$. For $P \in U$, we define the \emph{fractional jump index of $\Psi$ at $P$} as
\begin{equation*}
\mathfrak{J}_P = \min \set{k \in \Z_{\geq 1} \, : \, \Psi^k (P) \in U}.
\end{equation*}

The \emph{fractional jump of $\Psi$} is then defined as the map
\begin{equation*}
\psi : \A^n \rightarrow \A^n, \quad x \mapsto \pi^{-1} \Psi^{\mathfrak{J}_{\pi(x)}} \pi (x).
\end{equation*}

Essentially, the map $\psi$ is defined on a point $x \in \A^n$ as follows: we firstly send $x$ in $\PP^n$ via the canonical map $\pi$, then we iterate $\Psi$ on $\pi (x)$ until we end up with a point in $U$, and finally we take its image in $\A^n$ via $\pi^{-1}$.

When $\Psi$ acts transitively on $\PP^n$, its fractional jump $\psi$ admits an explicit description in terms of multivariate linear fractional transformations. More precisely, we have the following:

\begin{theorem}[{\cite[Section 5]{bib:AGM17}}] \label{explicit_form}
Let $\Psi$ be a transitive automorphism of $\PP^n$, and let $\psi$ be its fractional jump. Then, for $i \in \set{1, \ldots, n+1}$ there exist
\[ a_1^{(i)}, \ldots, a_n^{(i)}, b^{(i)} \in \F_q[x_1, \ldots, x_n] \]
of degree $1$ such that, if
\begin{align*}
U_1 &= \set{x \in \A^n \, : \, b^{(1)} (x) \neq 0}, \\
U_i &= \set{x \in \A^n \, : \, b^{(i)} (x) \neq 0, \text{ and } b^{(j)} (x) = 0, \, \forall j \in \set{1, \ldots, i-1}}, \\ &\text{for } i \in \set{2, \ldots, n+1}, \\
&\text{and} \\
f^{(i)} &= \bigg( \frac{a_1^{(i)}}{b^{(i)}}, \ldots, \frac{a_n^{(i)}}{b^{(i)}} \bigg),  \\
&\text{for } i \in \set{1, \ldots, n+1},
\end{align*}
then $\psi (x) = f^{(i)} (x)$ if $x \in U_i$. Moreover, the rational maps $f^{(i)}$ can be explicitly computed.
\end{theorem}

\begin{remark}
The reader should notice that the $b^{(i)}$ are equal on each component, and therefore the evaluation of $\psi$ only requires one inversion in the base field.
\end{remark}

\begin{remark}\label{Uideflastrows}
Another important fact to notice is that the definition of $\psi$ depends uniquely on the rows of $M^i$'s, where $M\in \GL_{n+1}(\F_q)$ is any matrix in the class of $\Psi$. In fact, notice that if the last row of $M^i$ is 
$(m^{(i)}_{n+1, 1},\dots ,m^{(i)}_{n+1,n+1})$, then  
$b^{(i)}= m^{(i)}_{n+1,n+1}+\sum^{n}_{j=1}{m^{(i)}_{n+1,j} x_j }$. 
On the other hand, for any $j\in \{1,\dots, n\}$, if $(m^{(i)}_{j,1},\dots ,m^{(i)}_{j,n+1})$ is the $j$-th row of $M^i$, then $a_j^{(i)}=m^{(i)}_{j,n+1}+\sum^{n}_{j=1}{m^{(i)}_{j,n+1} x_j }$.
What is done here is essentially dehomogenising the projective map induced by the class of $M^i$ and then restricting that to the affine points.
\end{remark}

We now provide a simple example to fix the ideas.

\begin{example}
Let $q = 5$ and $n = 2$. Consider the automorphism of $\PP^2$ defined by
\[ \Psi ([X_0 : X_1 : X_2]) = [3 X_0 + 2 X_1 + X_2 : 3 X_0 + 3 X_1 + X_2 : 3 X_1 + 4 X_2 ].\]
A representative for $\Psi$ in $\GL_3 (\F_5)$ is given by
\[ M = \begin{pmatrix}
3 & 2 & 1 \\
3 & 3 & 1 \\
0 & 3 & 4
\end{pmatrix},\]
whose characteristic polynomial
\[ \chi_M (T) = T^3 + 4 T + 3\]
is projectively primitive, since it is irreducible, and $(5^3-1)/(5-1) = 31$ is prime. By \cite[Theorem 3.4]{bib:AGM17}, it follows that $\Psi$ acts transitively on $\PP^2$, and then Theorem \ref{explicit_form} applies to the fractional jump $\psi$ of $\Psi$. Direct computations show that for
\begin{align*}
U_1 &= \set{(x_1, x_2) \in \A^2 \, : \, 3 x_2 + 4 \neq 0}, \\
U_2 &= \set{(x_1, x_2) \in \A^2 \, : \, 3x_2 + 4 = 0, \text{ and } 4 x_1 + x_2 + 4 \neq 0}, \\
U_3 &= \set{(1, 2)},
\end{align*}
and
\begin{align*}
f^{(1)} (x_1, x_2) &= \bigg( \frac{3x_1 + 2 x_2 + 1}{3 x_2 + 4}, \frac{3 x_1 + 3 x_2 + 1}{3 x_2 + 4} \bigg), \\
f^{(2)} (x_1, x_2) &= \bigg( \frac{4}{4 x_1 + x_2 + 4}, \frac{3 x_1 + 3 x_2}{4 x_1 + x_2 + 4} \bigg), \\
f^{(3)} (x_1, x_2) &= \bigg( \frac{2 x_2 + 1}{3 x_2 + 1}, \frac{3 x_1 + 1}{3 x_2 + 1} \bigg),
\end{align*}
we have that $\set{U_i }_{i \in \set{1, 2, 3}}$ is a disjoint covering of $\A^2$ such that $\psi (x) = f^{(i)} (x)$ if $x \in U_i$.
\end{example}

The purpose of this section is to show that transitive fractional jumps can only arise from transitive projective automorphisms, except from some very special cases, which can be entirely classified. Before proving the main theorem, let us recall a standard linear algebra fact, which follows from the results in \cite[XIV, \S2, \S3]{bib:lang02}.

\begin{lemma} \label{existence_cyclic_vector}
Let $\Bbbk$ be a field, let $V$ be a finite dimensional vector space over $\Bbbk$, and let $M$ be a $\Bbbk$-linear endomorphism of $V$. Assume that the minimal polynomial and the characteristic polynomial of $M$ are equal. Then, there exists $v_0 \in V$ such that the set $\set{M^k v_0 \, : \, k \in \Z_{\geq 0}}$ spans $V$ over $\Bbbk$.
\end{lemma}

We also need the following lemma:

\begin{lemma} \label{orders_in_quotients}
Let $p(T) \in \F_q [T]$ be an irreducible polynomial, and let $e \geq 1$ be a positive integer. Let $[T]$ be the class of $T$ in $\Gamma = (\F_q[T] / (p(T)^e))^*$, and let $[[T]]$ be the class of $T$ in $G = \Gamma / \F_q^*$. Then, the order of $[[T]]$ in $G$ equals the order of $[T]^{q-1}$ in $\Gamma$.
%Let $p(T) \in \F_q [T]$ be an irreducible polynomial, and let $e \geq 1$ be a positive integer. Let $\Gamma = (\F_q[T] / (p(T)^e))^*$, and let $G = \Gamma / \F_q^*$. Also, for every $f(T) \in \F_q[T]$ let $[f(T)]$ and $[[f(T)]]$ be its classes in $\Gamma$ and $G$ respectively. Then, the order of $[[T]]$ in $G$ equals the order of $[T]^{q-1}$ in $\Gamma$.
\end{lemma}

\begin{proof}
Let $k$ be the order of $[[T]]$ in $G$ and let $h$ be the order of $[T]^{q-1}$ in $\Gamma$. Then, $[[T]]^k = 1$ in $G$ gives $[T]^k \in \F_q^*$. But then $1 = ([T]^k)^{q-1} = ([T]^{q-1})^k$, and so $h \mid k$.

On the other hand, let us firstly show that if $s \in \F_q[T] / (p(T)^e)$ satisfies $s^{q-1}-1 = 0$, then $s \in \F_q^*$. In fact, by reducing $s$ modulo $p(T)$ we get that
\[s = c + k(T) p(T) \mod p(T)^e, \quad \text{ for } c \in \F_q^* \text{ and } k(T) \in \F_q[T]. \]
Now, by multiplying the equation $s^{q-1} - 1 = 0$ by $s$, and plugging in the above special form for $s$, we get
\begin{align*}
(c + k(T) p(T))^q - (c + k(T) p(T)) &\equiv (k(T) p(T))^q - k(T) p(T) \\
&\equiv k(T) p(T) ((k(T) p(T))^{q-1} - 1) \equiv 0 \mod p(T)^e.
\end{align*}
But now $k(T) p(T))^{q-1} - 1$ is invertible modulo $p(T)^e$, and so $k(T) p (T)$ must be zero modulo $p(T)^e$, which forces $s$ to be $c$ modulo $p(T)^e$.

It then follows that $1 = ([T]^{q-1})^h = ([T]^h)^{q-1}$ in $\Gamma$ gives $[T]^h \in \F_q^*$, from which we get $[[T]]^h = 1$ in $G$, and so $k \mid h$.
\end{proof}

The main result of this section is the following:

\begin{theorem}
Let $\Psi$ be an automorphism of $\PP^n$ and let $\psi$ be its fractional jump. Then, $\Psi$ acts transitively on $\PP^n$ if and only if $\psi$ acts transitively on $\A^n$, unless $q$ is prime and $n = 1$, or $q = 2$ and $n = 2$, with explicit examples in both cases.
%Let $\Psi$ be an automorphism of $\PP^n$ and let $\psi$ be its fractional jump. Let $\Psi = [M] \in \PGL_{n+1} (\F_q)$ for some $M \in \GL_{n+1} (\F_q)$, and let $\mu_M (T) \in \F_q [T]$ be the minimal polynomial of $M$. Assume that $\mu_M (T)$ is square-free. Then, $\Psi$ acts transitively on $\PP^n$ if and only if $\psi$ acts transitively on $\A^n$. %, unless $q$ is prime and $n = 1$, or $q = 2$ and $n = 2$, with explicit examples in both cases.
\end{theorem}

\begin{proof}
For any $q$ and $n$, it is immediate to show that if $\Psi$ is transitive then $\psi$ is transitive.
In the case of $q$ prime and $n = 1$ or $q = 2$ and $n = 2$ there exist explicit examples of transitive affine transformations, namely
\begin{align*}
\varphi_1 (x) &= x + 1, \quad &\text{if } q \text{ is prime and } n=1, \\
\varphi_2 (x_1, x_2) &= \begin{pmatrix} 1 & 1 \\ 0 & 1 \end{pmatrix} \cdot \begin{pmatrix} x_1 \\ x_2 \end{pmatrix} + \begin{pmatrix} 1 \\ 1 \end{pmatrix}, \quad &\text{if } q=2 \text{ and } n=2.
\end{align*}
Define then
\begin{align*}
\Phi_1 ([X_0 : X_1]) &= [X_0 + X_1 : X_1], \quad &\text{if } q \text{ is prime and } n=1, \\
\Phi_2 ([X_0 : X_1 : X_2]) &= [X_0 + X_1 + X_2 : X_1 + X_2 : X_2], \quad &\text{if } q=2 \text{ and } n=2.
\end{align*}
Clearly, $\varphi_i$ is the fractional jump of $\Phi_i$ for $i \in \set{1, 2}$. However, it is immediate to see that $\Phi_i$ fixes the hyperplane at infinity, so cannot be transitive for $i \in \set{1, 2}$.

%It is immediate to show that if $\Psi$ acts transitively on $\PP^n$ then $\psi$ acts transitively on $\A^n$.

Let us now assume that we are not in the above pathological cases, and that $\psi$ is transitive. Write $\Psi = [M] \in \PGL_{n+1} (\F_q)$ for some $M \in \GL_{n+1} (\F_q)$, and let $\chi_M (T), \mu_M (T) \in \F_q [T]$ be respectively the characteristic polynomial and the minimal polynomial of $M$. The vector space $V = \F_q^{n+1}$ over $\F_q$ has a natural structure of $\F_q[T]$-module given by
\[ f(T) v = f(M) v, \quad \text{for } f(T) \in \F_q [T], \text{ and } v \in V.\]
Let $\F_q [M]$ be the subalgebra of the algebra of $\F_q$-linear endomorphisms of $V$ generated by $M$, and let $G_\Psi$ be the quotient (multiplicative) group $\F_q [M]^* / \F_q^*$.
%and $\mu_M (T)$ is square-free.
%that $\psi$ acts transitively on $\A^n$.

%Let $\Psi = [M] \in \PGL_{n+1} (\F_q)$ for some $M \in \GL_{n+1} (\F_q)$, and let $\mu_M (T) \in \F_q [T]$ be the minimal polynomial of $M$. Let $\chi_M (T) \in \F_q [T]$ be the characteristic polynomial of $M$.
We firstly prove that $\mu_M (T) = \chi_M (T)$. Assume by contradiction $\mu_M (T) \neq \chi_M (T)$, so that $\deg \mu_M (T) \leq n$. Then, given any $P \in U$, and any $x \in \A^n$ such that $P = \pi (x)$, we have
\begin{align*}
q^n = o_\psi (x) &\leq o_\Psi (P) \\
&\leq |G_\Psi| \\
&\leq \frac{q^n - 1}{q-1} < q^n,
\end{align*}
a contradiction, which implies $\mu_M (T) = \chi_M (T)$.

%Write
%\[ \mu_M (T) = \prod_{i = 1}^\ell p_i (T),\]
%for $\ell  \in \Z_{\geq 1}$, and $p_i (T) \in \F_q [T]$ irreducible coprime polynomials. % of degree $f_i \in \Z_{\geq 1}$, and $e_i  \in \Z_{\geq 1}$, so that $\sum_{i = 1}^\ell f_i e_i = n+1$. Let us assume $e_1 \geq \cdots \geq e_\ell$.

Define now
\[N = \set{P \in H \, : \, \Psi^i (P) \in H, \, \forall i \in \Z}.\]
We want to show that $N = \emptyset$. Note that this would immediately imply that $\Psi$ is transitive. To see this, given any $P, Q \in \PP^n$, if $N = \emptyset$ then there exist $i, j \in \Z$ such that $P' = \Psi^i (P), Q' = \Psi^j (Q) \in U$. Let $x', y' \in \A^n$ be such that $P' = \pi (x')$ and $Q' = \pi (y')$. As $\psi$ acts transitively on $\A^n$ by hypothesis, there exists $\ell \in \Z$ such that $y' = \psi^\ell (x')$. Then, by the definition of $\psi$, there exists an integer $k \geq \ell$ such that $Q' = \Psi^k (P')$. In conclusion, we get $Q = \Psi^{i+k-j} (P)$, and so we have that if $N = \emptyset$ then $\Psi$ is transitive.

Assume by contradiction that $N \neq \emptyset$. Define
\[ W = \set{v \in V \, : \, (M^i v)_{n+1} = 0, \, \forall i \in \Z},\]
where $(M^i v)_{n+1}$ denotes the $(n+1)$-th component of $M^i v$. It is immediate to check that $W$ is a subspace of $V$, and that $N = \PP W$. Also, $W$ is clearly $\F_q [M]$-invariant, and so it is an $\F_q[T]$-submodule of $V$. %, with respect to the usual action of $\F_q[T]$ on $V$ given by
Let $g(T) \in \F_q [T]$ is a monic generator of the annihilator $\mathrm{Ann}_{\F_q [T]} (W)$ of $W$ as $\F_q [T]$-module. We have that $g(T) \mid \mu_M (T)$, since $\mu_M (M) w = 0$ for any $w \in W$. Also, $g(T) \neq 1$ as $N \neq \emptyset$ by assumption, and $g(T) \neq \mu_M (T)$, since $N \subseteq H$ gives $\deg g(T) \leq n$. This shows that if $N \neq \emptyset$ the $\mu_M (T)$ is reducible.
%We want to prove that $\mu_M (T)$ is irreducible. In this case, in fact, if , then . Also, , and so $g(T) = \mu_M (T)$, which gives a contradiction, as. This shows that if $\mu_M (T)$ is irreducible then $N = \emptyset$.

Let us now prove instead that $\mu_M (T)$ is irreducible, so that we get a contradiction. We firstly prove that $\mu_M (T) = p(T)^e$ for some irreducible polynomial $p(T) \in \F_q [T]$ and some integer $e \geq 1$.

Since $\mu_M (T) = \chi_M (T)$, then by Lemma \ref{existence_cyclic_vector} we know that there exists $v_0 \in V$ such that the set $\set{M^k v_0 \, : \, k \in \Z_{\geq 0}}$ spans $V$ over $\F_q$. Clearly, $v_0 \notin W$, since otherwise we would have $W = V$, as $W$ is $\F_q[M]$-invariant, which is a contradiction as $N \subseteq H$. We show now that $d(M) v_0 \in W \setminus \set{0}$ for any $d (T) \in \F_q[T]$ such that $d(T) \mid \mu_M (T)$, and $d(T) \neq 1, \mu_M (T)$. Let $d(T)$ be any of such polynomials. Clearly $d(M) v_0 \neq 0$, as otherwise the span of $\set{M^k v_0 \, : \, k \in \Z_{\geq 0}}$ over $\F_q$ would have dimension less or equal than $\deg d(T)$, which is less or equal than $n$ by assumption. Define then $W_d$ to be the span of $\set{M^k d(M) v_0 \, : \, k \in \Z_{\geq 0}}$ over $\F_q$. It is immediate to see that $W_d$ is an $\F_q[M]$-invariant subspace of $V$ of dimension less or equal than $\deg (\mu_M (T) / d(T))$, which is less or equal than $n$ by assumption. Assume by contradiction $d(M) v_0 \notin W$. Then, if we let $P_d$ be the class of $d(M) v_0$ in $\PP^n$, we have $P_d \notin N$, and so there exists $i \in \Z$ such that $Q_d = \Psi^i (P_d) \in U$. Let $y_d \in \A^n$ be such that $Q_d = \pi (y_d)$. Then, 
\begin{align*}
q^n = o_\psi (y_d) &\leq o_\Psi (Q_d) \\
&= |\mathcal{O}_\Psi (P_d)| \\
&\leq |\PP W_d| \\
&\leq \frac{q^n - 1}{q-1} < q^n,
\end{align*}
a contradiction. This proves that $d(M) v_0 \in W \setminus \set{0}$ for any $d (T) \in \F_q[T]$ such that $d(T) \mid \mu_M (T)$, and $d(T) \neq 1, \mu_M (T)$.

Recall that we want to prove that $\mu_M (T) = p(T)^e$ for some irreducible polynomial $p(T) \in \F_q [T]$ and some integer $e \geq 1$. Assume then by contradiction that there exist $p_1 (T), p_2 (T) \in \F_q[T]$ distinct irreducible polynomials such that $p_1(T), p_2 (T) \mid \mu_M (T)$. Then, by B\'ezout's identity, there exist $a(T), b(T) \in \F_q [T]$ such that $a (T) p_1 (T) + b(T) p_2 (T) = 1$, and so $a(M) p_1 (M) v_0 + b(M) p_2 (M) v_0 = v_0$. Now, $p_i (M) v_0 \in W$ for $i \in \set{1, 2}$ by the claim above, and so $v_0 \in W$, as $W$ is an $\F_q [M]$-invariant subspace of $W$, which is a contradiction. Therefore, we conclude that $\mu_M (T) = p(T)^e$ for some irreducible $p(T) \in \F_q [T]$ and some $e \geq 1$.

We finally show that $\mu_M (T)$ is irreducible, that is $e = 1$. Let us set $f = \deg p(T)$, and let $[[ T ]]$ be the class of $T$ in $G_\Psi$. We want to show that the order of $[[T]]$ in $G_\Psi$ divides
\[ A(q, e, f) = q^{\lceil \log_q e \rceil}\frac{q^f-1}{q-1}. \]
Let $[T]$ be the class of $T$ in $\F_q[M]^*$. As $\F_q[M]^* \cong (\F_q [T] / (p(T)^e))^*$, by Lemma \ref{orders_in_quotients} it is enough to show that the order of $[T]^{q-1}$ in $\F_q [M]^*$ divides $A(q, e, f)$. Now, since $[T]^{q^f-1} \equiv 1 \mod p(T)$, we have $[T]^{q^f-1} = 1 + k(T) p(T)$ for some $k(T) \in \F_q [T]$, and so
\begin{align*}
([T]^{q-1})^{A(q, e, f)} &= ([T]^{q^f-1})^{q^{\lceil \log_q e \rceil}} \\
&= [1 + k(T)p(T)]^{q^{\lceil \log_q e \rceil}} \\
&= [1 + k(T)^{q^{\lceil \log_q e \rceil}} p(T)^{q^{\lceil \log_q e \rceil}}] = 1 \quad \text{in } \F_q[M]^*,
\end{align*}
as $q^{\lceil \log_q e \rceil} \geq e$.

Let $P \in U$, and let $x \in \A^n$ be such that $P = \pi (x)$. Then
\begin{align*}
q^n = o_\psi (x) &\leq o_\Psi (P) \\
&\leq A(q, e, f),
\end{align*}
since the size of $\mathcal{O}_\Psi (P)$ is less or equal than the order of $[[T]]$ in $G_\Psi$. Notice also that here $n = ef-1$, since $\mu_M (T) = p(T)^e$ and $f = \deg p(T)$.

Assume by contradiction that $e \geq 2$. We firstly prove that this forces $f = 1$. Rewrite the inequality $q^{ef-1} \leq A(q, e, f)$ as
\begin{equation} \label{inequality_transitivity}
q^{ef-1 - \lceil \log_q e \rceil} \leq \frac{q^f-1}{q-1}.
 \end{equation}
 Since the quantity
 \[ q^{ef-1 - \lceil \log_q e \rceil} - \frac{q^f-1}{q-1} \]
 is increasing in $e$ and $f$, it is enough to show that \eqref{inequality_transitivity} is never verified for $e = 2$ and $f = 2$. Now, inequality \eqref{inequality_transitivity} for $e = 2$ and $f = 2$ becomes
 \[ q^2 \leq q+1, \]
 which is false for every $q$. Then $f = 1$.
 
 We want now to show that for $f = 1$ the inequality \eqref{inequality_transitivity} forces $q$ to be prime and $n = 1$, or $q = 2$ and $n = 2$, which are exactly the pathological cases we excluded. For $f = 1$, inequality \eqref{inequality_transitivity} becomes
 \[ q^{e-1 - \lceil \log_q e \rceil} \leq 1, \]
 which is equivalent to
 \[ e - 1 - \lceil \log_q e \rceil \leq 0. \]
 The quantity $ e - 1 - \lceil \log_q e \rceil$ is clearly increasing in $e$. Then, for $e \geq 4$ it is enough to show that it never holds for $e = 4$. In this case, in fact, we have $\lceil \log_q 4 \rceil \leq 2$ for every $q$, and so $4 - 1 - \lceil \log_q 4 \rceil \geq 1$ for every $q$. For $e = 3$, in which case $n = 2$, we have $\lceil \log_2 3 \rceil = 2$, and $\lceil \log_q 3 \rceil = 1$ otherwise. Then, the inequality is satisfied for $q = 2$, and never satisfied for $q \neq 2$. Finally, for $e = 2$ we have $n = 1$. Since for $n = 1$ if $\Psi$ sends a point of $U$ to the point at infinity, then $\psi$ transitive gives $\Psi$ transitive by \cite[Proposition 2.6]{bib:AGM17}, and so $e = 1$ by \cite[Theorem 3.4]{bib:AGM17}, a contradiction. We have then that $\Psi$ maps no point of $U$ to the point at infinity, and so $\psi$ is an affine map. But then, since $\psi$ is transitive (and in particular the inequality holds) then $q$ is prime by \cite[Theorem 2.7]{bib:AGM17}. In conclusion, we proved that if $e \geq 2$ then $q$ is prime and $n = 1$, or $q = 2$ and $n = 2$, which are the pathological cases excluded at the beginning. Therefore $e = 1$, and so $\mu_M (T)$ is irreducible.
\end{proof}

%that is there exists an $\F_p$-vector subspace $W$ of $V$ such that $N = \PP W$. Also, $W$ is clearly $M$-invariant, and so it is an $\F_p [T]$-submodule of $V$. Let $g(T) \in \F_q [T]$ be a generator of the annihilator $\mathrm{Ann}_{\F_q [T]} (W)$ of $W$. Clearly $g(T) \mid \mu_M (T)$, as $\mu_M (M) w = 0$ for any $w \in W$. Also, $g(T) \neq 1$ as $N \neq \emptyset$ by assumption, and $g(T) \neq \mu_M (T)$ as $N \subseteq H \neq \PP^n$. If $\mu_M (T)$ is irreducible we get a contradiction. Otherwise, as $\mu_M (T)$ is square-free by assumption, there exists $h(T) \mid \mu_M (T)$, coprime to $g(T)$, and such that $h(T) \neq 1$. There then exists a non-zero $\F_p$-vector subspace $W'$ of $V$ which is $M$-invariant, has dimension $\dim_{\F_p} W' \leq \deg h(T)$, and satisfies $W \cap W' = \set{0}$. Let $v \in W'$ be non-zero, and let $P = [v]$ be its class in $\PP^n$. As $W \cap W' = \set{0}$, there exists $i \in \Z$ such that $\Psi^i (P) \in U$. Let $Q = \Psi^i (P)$, and let $y \in \A^n$ such that $Q = \pi (y)$. Then
%\begin{align*}
%q^n = o_\psi (y) &\leq o_\Psi (Q) \\
%&\leq |U| \\
%&\leq \frac{q^{\deg h(T)} - 1}{q-1}  \\
%&\leq \frac{q^{n} - 1}{q-1} < q^n,
%\end{align*}
%a contradiction.
%
%Let us now finally prove that $\Psi$ acts transitively on $\PP^n$. Let

%%%%%%%%%%%%%%%%%%%%%%%%%%%%%%%%%%%%%%%%%%%%%%%%%%%
%%%%%%%%%%%%%%%%%%%%%%%%%%%%%%%%%%%%%%%%%%%%%%%%%%%

\section{Artin-Schreier fractional jumps}

Let $q = p$ be a prime number. In this section we consider fractional jumps of automorphisms of $\PP^{p-1}$ defined by companion matrices of Artin-Schreier polynomials
\[ \alpha_{c} (T) = T^p-T-c \in \F_p[T], \quad \text{for } c \in \F_p^*.\]

\begin{proposition}
The polynomial $\alpha_c (T)$ is projectively primitive for every $c \in \F_p^*$.
\end{proposition}

\begin{proof}
Notice that it is well known that the polynomials $\alpha_c(T)$ are irreducible for every $c \in \F_p^*$ by the theory of Artin-Schreier extensions. Let now $c \in \F_p^*$ be fixed. We want to show that $\alpha_c (T)$ is projectively primitive. Let $c' \in \F_p^*$ be such that $c/c'$ generates $\F_p^*$. Then, the polynomial $T^p-T-c/c'$ is primitive by \cite[Theorem 1.2]{bib:cao10}, and so projectively primitive. Now, this implies that the polynomial $c' T^p - c' T - c = (c' T)^p - c'T - c$ is projectively primitive, and so $\alpha_c (T)$ is projectively primitive.
\end{proof}

Fix $c \in \F_p^*$, let $M \in \GL_p (\F_q)$ be the companion matrix of $\alpha_c (T)$, let $\Psi = [M]$, and let $\psi$ be the fractional jump of $\Psi$. Let $x_0 \in \A^{p-1}$, and let $\{ x^{(k)} \}_{k \in \N}$ be the sequence recursively defined by $x^{(k+1)} = \psi (x^{(k)})$. By \cite[Theorem 3.4]{bib:AGM17} we know that the sequence $\{ x^{(k)} \}_{k \in \N}$ has full orbit.

\subsection{Explicit description} 
In what follows we want to give the explicit description of the Artin-Schreier fractional jump $\psi$.

For $i \in \set{1, \ldots, p-1}$ we have that
\[ M^i= \left( \begin{array}{c|c}
0_{i, p-i} & J_i (c)^t \\
\hline
\mathrm{Id}_{p-i} & E_{p-i, i}^{(1, i)}
\end{array} \right), \]
where $0_{i, p-i}$ is the $i \times (p-i)$ zero matrix, $J_i (c)^t$ is the transpose of a Jordan block of size $i \times i$ and eigenvalue $c$, that is
\[ J_i(c)^t = \begin{pmatrix}
                 c & 0 & \cdots &  \cdots & 0 & \\
		1 & c & 0 &   &  \vdots & \\
		0 & 1 & c &    & \vdots \\
		\vdots &  & \ddots  & \ddots  & 0\\
		0 &\cdots & 0 & 1 &  c &
                \end{pmatrix}, \]        
the matrix $\Id_{p-i}$ is the $(p-i) \times (p-i)$ identity, and $E_{p-i, i}^{(1, i)}$ is the $(p-i) \times i$ matrix with $(1, i)$-entry equal to $1$, and all the other entries equal to zero. For $i = p$ we clearly have $M^p = M + c \mathrm{Id}_p$. 

%For $i = p$ we have that
%\[ M^p = J_p (c)^t + \begin{pmatrix}
%                 0 & \cdots & 0 & c \\
%                 0 & \cdots & 0 & 1 \\
%                 0 & \cdots & \cdots & 0 \\
%                 \vdots & & & \vdots \\
%                 0 & \cdots & \cdots & 0
%                \end{pmatrix}. \]

Following Remark \ref{Uideflastrows}, let us now compute explicitly the polynomials $b^{(i)}$'s and the sets $U_i$'s, for $i \in \set{1, \ldots, p}$, by looking at the last row of $M^i$. 
\begin{align*}
b^{(i)} &= x_{p-i}, \\
&\text{for } i \in \set{1, \ldots, p-2}, & \\
b^{(p-1)} &= x_1 + 1, \\
b^{(p)} &= x_{p-1} + c, %& \text{for } x \in \A^{p-1},
\end{align*}
which gives
\begin{align*}
U_1 &= \set{x \in \A^{p-1} \, : \, x_{p-1} \neq 0}, \\
U_i &= \set{x \in \A^{p-1} \, : \, x_{p-i} \neq 0, \text{ and } x_{p-j} = 0, \, \forall j \in \set{1, \ldots, i-1}}, \\
&\text{for } i \in \set{2, \ldots, p-2}, \\
U_{p-1} &= \set{x \in \A^{p-1} \, : \, x_1+1 \neq 0, \text{ and } x_{p-j} = 0, \, \forall j \in \set{1, \ldots, p-2}}, \\
U_p &= \set{x \in \A^{p-1} \, : \, x_{p-1} + c \neq 0, \, x_1 + 1 = 0, \text{ and } x_{p-j} = 0, \, \forall j \in \set{1, \ldots, p-2}} \\
&= \set{(-1, 0, \ldots, 0)}.
\end{align*}

The polynomials $a_j^{(i)}$, for $i \in \set{1, \ldots, p}$ and $j \in \set{1, \ldots, p-1}$, are easily computed as well by looking at the $j$-th row of $M^i$.

\begin{itemize}
\item for $i=1$ we have that
\begin{itemize}
\item if $j=1$ then $a_1^{(1)}=c$, 
\item if $j=2$ then $a_2^{(1)}=x_1 + 1$, 
\item for any $j\in \{3,\dots, p-1\}$ then $a_j^{(1)}=x_{j-1}$.
\end{itemize}  
\item for $i\in \{2,\dots p-1\}$ we have that
\begin{itemize}
\item if $j=1$ then $a_1^{(i)}=cx_{p-i+1}$,
\item if $j\in \{2,\dots, i-1\}$ then $a_j^{(i)}=x_{p-i+j-1}+cx_{p-i+j}$,
\item if $j = i$ then $a_i^{(i)} = x_{p-1} + c$,
\item if $j=i+1$ then $a_{i+1}^{(i)}= x_1+ 1$,
\item if $j\in \{i+2,\dots, p-1\}$ then $a_j^{(i)}= x_{j-i}$.
\end{itemize} 
\item for $i=p$ we have that
\begin{itemize}
\item if $j=1$ then $a^{(p)}_{1}=cx_1+ c$,
\item if $j=2$ then $a^{(p)}_{2}=x_1+ c x_2 + 1$,
\item if $j\in \{3,\dots,p-1\}$ then $a^{(p)}_{j}=x_{j-1} + c x_j$.
\end{itemize}
\end{itemize}

By Theorem \ref{explicit_form} this provides the explicit structure of $\psi$.

\subsection{Computational complexity}

Now that we have the explicit description of the fractional jump, we are ready to  establish the expected complexity of computing a random term in the sequence $\{ x^{(k)} \}_{k \in \N}$ given by iterating the Artin-Schreier fractional jump $\psi$.

The expected complexity of computing  $x^{(k+1)}$ given a term $x^{(k)}$ chosen uniformly at random in the sequence is
\[ \mathbb{E}= \sum_{i = 1}^{p} p_i c_i,\]
where $p_i$ is the probability that $x^{(k)} \in U_i$, which is
\[ p_i = \begin{cases}
p^{-i} (p-1), & \text{if } i \in \set{1, \ldots, p-1}, \\
p^{1-p}, & \text{if } i = p,
\end{cases}\]
and $c_i$ is the complexity of evaluating $\psi$ at $x^{(k)}$ when $x^{(k)} \in U_i$.

We want now to evaluate $c_i$ for $i \in \set{1, \ldots, p}$. If $x^{(k)} \in U_i$, the number of sums needed to compute $x^{(k+1)} = \psi(x^{(k)}) = f^{(i)} (x^{(k)})$ is $s_i=\sum^p_{j=1} (r^{(i)}_j-1)$, where $r^{(i)}_j$ is  the number of  non-zero entries in the $j$-th row of the matrix $M^i$.

Since the denominators of the components of $f^{(i)}$ are all equal, the number of inversions needed is always $1$.

 Also, the number of multiplications needed is given by the number $m_i$ of entries different from $0$ and $1$ in the $p\times (p-1)$ submatrix of $M^i$ given by dropping the last column (this can be seen as the last component of the projective point is set to $1$ in the fractional jump) plus the number of multiplications of $b^{(i)}(x^{(k)})^{-1}$ by the $a_j^{(i)}$'s, which is simply $p-1$.
 
Since the length of the orbit $p^{p-1}$ is superexponential, the size of $p$ can be chosen relatively small in such a way that one can build look-up tables for the operations in $\mathbb F_p$ (so they will all have the same cost) and still get a huge orbit. Therefore

%Given the considerations above, all the computations can be performed with look-up tables, and so they all count the same. Therefore 
\[c_i=\underbrace{s_i}_{\text{sums}}+\underbrace{1}_{\text{inversions}}+\underbrace{m_i+p-1}_{\text{multiplications}}.\]

It remains to compute $s_i$ and $m_i$. Given the explicit description previously provided, we have $s_i=i$ for $i \in \set{1, \ldots, p-1}$ and $s_p = p+1$, and $m_i=i-1$ for $i \in \set{1, \ldots, p}$. Therefore, we have $c_i=p+2i-1$ for $i \in \set{1, \ldots, p-1}$ and $c_p = 3p$. %In order to do so, we compute explicitely $M^i$. 

The expected complexity is then
\begin{align*}
\mathbb{E} &= 3 p^{2-p} + \sum_{i = 1}^{p-1} p^{-i} (p-1)(p+2i-1) \\
&= 3p^{2 - p} - \frac{3p^3 - (p^2 + 1)p^p - 4p^2 + 3p}{p^p(p - 1)} = p + O \bigg( \frac{1}{p} \bigg).
\end{align*}

This means that the expected complexity of computing the $(k+1)$-th vector of the sequence roughly consists of $p$ checks of the look-up tables, one for each component: morally, we are filling out each component of the term of the sequence by directly reading the look-up table, which is why the process is very efficient.

\begin{remark}
Clearly, the expected complexity can be further optimised by using the equations defining the $U_i$'s, but this will not affect the asymptotic behaviour of $\mathbb{E}$.
\end{remark}

%%%%%%%%%%%%%%%%%%%%%%%%%%%%%%%%%%%%%%%%%%%%%%%%%%%
%%%%%%%%%%%%%%%%%%%%%%%%%%%%%%%%%%%%%%%%%%%%%%%%%%%

\section{Conclusions and further research}

In this paper we proved that the transitivity of the fractional jumps and the transitivity of the projective automorphisms inducing them are equivalent conditions, except from some degenerate cases which are entirely classified. This puts the last stone for the foundational theory of this new construction: for fixed base field and fixed dimension, the problem of finding all transitive fractional jump is now reduced to finding transitive projective automorphisms. 
In addition, using the theory of Artin-Schreier polynomials, we showed that the construction is sistematically feasible when the dimension of the projective space is prime and equal to the characteristic of the field.
The question now arising is:
\begin{question}
Can one give new explicit classes of projectively primitive polynomials?
\end{question}
Such new classes will allow to use companion matrices of such polynomials (or their conjugates) to build full orbit fractional jump sequences. In particular, it would be of interest to do this for any fixed dimension and in characteristic $2$, and with sparse polynomials.

%%%%%%%%%%%%%%%%%%%%%%%%%%%%%%%%%%%%%%%%%%%%%%%%%%%
%%%%%%%%%%%%%%%%%%%%%%%%%%%%%%%%%%%%%%%%%%%%%%%%%%%

\section*{Acknowledgment}

The authors are grateful to Andrea Ferraguti for preliminary reading of this manuscript, and for useful discussions and suggestions. The second author is thankful to the Swiss National Science Foundation grant number 171248.

%%%%%%%%%%%%%%%%%%%%%%%%%%%%%%%%%%%%%%%%%%%%%%%%%%%
%%%%%%%%%%%%%%%%%%%%%%%%%%%%%%%%%%%%%%%%%%%%%%%%%%%

\bibliographystyle{abbrv}
\bibliography{Bibliography.bib}

\end{document}